\documentclass[a4paper,11pt,french,openany]{article}
\usepackage[latin1]{inputenc}
\usepackage[T1]{fontenc}
\usepackage[francais]{babel}
\usepackage{amsfonts}
\usepackage{amscd}
\usepackage{graphicx}
\usepackage{babel}
\usepackage{fancyhdr}
\usepackage[a4paper,left=2cm,right=2cm,top=2cm,bottom=1.5cm]{geometry}
\usepackage[french]{minitoc}
\usepackage{tabularx}
\usepackage{amssymb}
\usepackage{amsthm}
\usepackage{multirow}
\usepackage{multicol}
\usepackage{array}
\usepackage[all]{xy}
\usepackage[centertags]{amsmath}
\usepackage{latexsym}
\usepackage{amsfonts}
\usepackage{array}
\newtheorem{theorem}{\textbf{Theorem}}[section]

\theoremstyle{plain}

\theoremstyle{plain}

\theoremstyle{plain}

\theoremstyle{plain}
\newtheorem{remarq}{\textbf{Remark}}[section]

\theoremstyle{plain}

\theoremstyle{plain}
\newtheorem{defn}{\textbf{Definition}}[section]
\theoremstyle{plain}

\theoremstyle{plain}

\theoremstyle{plain}

\theoremstyle{plain}

\theoremstyle{plain}

\theoremstyle{plain}

\begin{document}
\begin{center}
\textbf{Semi-global solutions to the Goursat problem for second-order hyper-quasilinear hyperbolic systems with lineary dependent principal coefficients
 and  applications to the vacuum  Einstein equations}
\end{center}
\begin{center}
 \textbf{  LOUOKDOM TAMTO Paul Giscard$^{1}$,  HOUPA
 DANGA Duplex Elvis$^{2}$  and KOUAKEP TCHAPTCHIE Yannick$^{3}$}
 \end{center}
 \begin{center}
 \begin{small}
$^{1}$Department of Mathematics, Faculty of Sciences, The University of
Ngaoundere, Dang, Ngaoundere, 454, ADAMAOUA, CAMEROON.\\
$^{2}$Department of Mathematics, Faculty of Sciences, The University of
Ngaoundere, Dang, Ngaoundere, 454, ADAMAOUA, CAMEROON.\\
$^{3}$Department of SFTI-EGCIM, The University of
Ngaoundere, Dang, Ngaoundere, 454, ADAMAOUA, CAMEROON.

\end{small}
\end{center}

\begin{center}
*Corresponding author(s). E-mail(s) : kouakep@aims-senegal.org ; \\
Contributing authors: elvishoupa@gmail.com; louokdompaulgiscard@gmail.com.
\end{center}
\textbf{Abstract} \par In this work, we significantly extend the results of  D.
Houpa, 2006 on the  Goursat problem for  second-order semi-linear hyperbolic  systems to the broader framwork of second-order
 hyper-quasilinear hyperbolic systems
of Goursat type, in which   the coefficients of the second-order derivatives
 depend linearly on the unknown.
  By adapting techniques inspired by   Y. Foures (Choquet)- Bruhat,
 Acta Mathematica, 1952. we show that in the Sobolev type spaces for the Goursat problem quasilinear hyperbolic of the second order considered, the solution exists and is defined
 in the vicinity of the meeting characteristic hypersurfaces which carry the initial data. As an application, in harmonic gauge, we derive a semi-global existence and uniqueness result for the
   vacuum Einstein equations.
 \begin{center}
\end{center}
 \section{Introduction}
For many years, the search for exact solutions of certain equations in mathematical physics (the vacuum Einstein equations, the Yang-Mills-Higgs equations, and the Maxwell equations) has represented a major challenge for physicists, due to their strong degree of nonlinearity.
To assist in this task, mathematicians have been led to impose additional conditions, known as gauge conditions (harmonic or temporal gauges for the Einstein equations, Lorenz gauge for the Yang-Mills-Higgs and Maxwell equations), which make it possible, in the limit, to establish results on local existence, semi-global existence \cite{2,4}, and even global existence \cite{ig}.

\par Motivated by these works, with the aim of solving the Goursat problem associated with the vacuum Einstein equations in harmonic gauge semi-globally in weighted Sobolev-type spaces, we first establish an existence and uniqueness result for semi-global solutions to hyper-quasilinear hyperbolic second-order Goursat problems in which the coefficients of the second-order derivatives depend linearly on the unknown function. This is done by closely following a method initiated by M. Dossa and S. Bah \cite{2} for the semi-global resolution of semilinear hyperbolic second-order Cauchy problems with initial data posed on a characteristic conoid, and further developed by D. Houpa and M. Dossa \cite{4} for the semi-global resolution of semilinear hyperbolic second-order Goursat problems. This approach relies on a careful combination of an existence and uniqueness result for semi-global solutions of the integral system arising from the Kirchhoff formulas associated with the problem; and a local existence and uniqueness result for the original problem, obtained via the local resolution of an associated mixed spacelike-characteristic second-order hyperbolic problem. \par The rest of this work is organized as follows. In Section 2, we formulate the problem, introduce the associated assumptions, and define a suitable geometric framework together with an appropriate functional setting for its analysis. In Section 3, using methods similar to those of Y. Foures (Choquet)-Bruhat \cite{cho}, we establish an existence and uniqueness result for semi-global solutions of the generalized Kirchhoff integral system associated with hyper-quasilinear second-order hyperbolic Goursat problems in which the coefficients of the second-order derivatives depend linearly on the unknown function $(\ref{8f})$. The problem is then solved by adopting a semilinearization approach as suggested by D. Houpa in his thesis prospects. This semilinearization method is based on a fixed-point argument, whose main ingredient is the existence and uniqueness result for solutions of the generalized Kirchhoff integral system associated with second-order semilinear Goursat problems established by D. Houpa and M. Dossa \cite{4}. In the same section, by using a local existence and uniqueness result for hyper-quasilinear second-order hyperbolic mixed spacelike-characteristic problems \cite{5}, we show that, within a Sobolev-type functional framework and under finite differentiability assumptions on the data without any smallness assumptions on the initial data and under fairly general structural assumptions on the nonlinear terms the solution of the considered problem is defined in a neighbourhood of the entire family of characteristic hypersurfaces carrying the initial data in Y. In Section 4, under the harmonic gauge condition, we apply this existence and uniqueness result for semi-global solutions to the vacuum Einstein equations. Section 5 is devoted to a discussion.  Finally, section 6 is devoted to declaration of competing interests.
\section{Geometric framework, assumptions and formulation of the problem}
We consider the following hyper-quasilinear system of equations:\\
$(G_{r}):~A^{\lambda\mu }(x^{\alpha},u_{s})\partial^{2}_{\lambda\mu}u_{r}+f_{r}(x^{\alpha},u_{s},\partial_{\nu}u_{s})=0 ~~~in ~~~ \mathcal{Y}$,
 where $\lambda,\alpha,\mu,\nu=0,...,3$; $r,s=1,...,n$;  $w=1,2$;
 $\partial^{2}_{\lambda\mu}=\frac{\partial^{2}}{\partial x^{\lambda}\partial x^{\mu}}$;  $\partial_{\nu}=\frac{\partial}{\partial x^{\nu}}$.\\
$\bullet$ $\mathcal{Y}=\lbrace (x^{\alpha}) \in \mathbb{R}^{4}/ x^{0}\geq \mid x^{1} \mid ; (x^{2},x^{3}) \in B\rbrace$, $S^{w}=\lbrace (x^{\alpha}) \in \mathbb{R}^{4}/ x^{0}=(-1)^{w-1} x^{1}; (x^{2},x^{3}) \in B\rbrace$, where $B$ is a closed bounded domain of $\mathbb{R}^{2}$;\\
$\bullet$ For any function $v$ defined in a domain of $\mathcal{Y}$, we set $[v]^{w}=v\mid_{S^{w}}$

 \begin{center}
                                                                                                        $[v]^{w}(x^{1},x^{2},x^{3})=v((-1)^{w-1}x^{1},x^{1},x^{2},x^{3})$, $(w=1,2)$
                                                                                                            \end{center}
$\bullet$ $\Gamma$ is a 2-surface of $\mathbb{R}^{4}$ defined by: $\lbrace (x^{\alpha}) \in \mathbb{R}^{4}: x^{0}=0= x^{1} ; (x^{2},x^{3}) \in B\rbrace$; $S=S^{1}\cup S^{2}$;  $\Gamma=S^{1}\cap S^{2}$; $D^{w}=\lbrace (x^{i}) \in \mathbb{R}^{3}: (x^{0},x^{i}) \in S^{w}\rbrace$, $\forall$ $T \in \mathbb{R}^{*}_{+}$, $\mathcal{Y}_{T}= \mathcal{Y}\cap \lbrace x^{0}\leq T\rbrace$,\\ $S^{w}_{T}=S^{w}\cap\lbrace x^{0}\leq T\rbrace$, $\sum^{w}_{T}=S^{w}\cap \lbrace x^{0}= T\rbrace$;\\
$\bullet$ $\forall$ $\sigma \in \mathbb{R}^{*}_{+}$, $P_{(\sigma)}$ is the hyperbolic of cylindrical equation: $ (x^{0})^{2}- (x^{1})^{2}=\sigma^{2}$,\\ $\mathcal{Y}_{(\sigma)}=\lbrace \in \mathbb{R}^{4}: \mid x^{1}\mid \leq x^{0} \leq \sqrt{(x^{1})^{2}+\sigma^{2}}\rbrace$, $G_{T,\sigma}=\mathcal{Y}_{(\sigma)}\cap \lbrace x^{0}= T\rbrace$,$\mathcal{Y}_{T,\sigma}=\mathcal{Y}_{T}\cap \mathcal{Y}_{(\sigma)}$;\\
$\bullet$ $\mathcal{Y}^{(f)}=\cup_{T \in \mathbb{R}^{*}_{+}}\mathcal{Y}_{T,f(T)}$ where $f$: $\mathbb{R}^{*}_{+}$ $\rightarrow$ $\mathbb{R}^{*}_{+}$ is an application;\\
$\bullet$ $\varphi_{r}^{w}$ is a function defined on $D^{w}$;\\ We consider the hyper-quasilinear system second order hyperbolic with initial data following
Goursat :

\begin{equation}\label{7}
\left\{
                \begin{array}{ll}
(G_{r}):~A^{\lambda\mu }(x^{\alpha},u_{s})\partial^{2}_{\lambda\mu}u_{r}+f_{r}(x^{\alpha},u_{s},\partial_{\nu}u_{s})=0 ~~~in ~~~ \mathcal{Y,} \hbox{} \\
~~~~~~~~~~~~~ u_{r}=\bar{u}_{r}=\varphi^{w}_{r} ~~on ~~S^{w};
                \end{array}
              \right.
\end{equation} In an appropriate functional framework involving Sobolev-type spaces to be specified, and under fairly general structural assumptions
   $(\mathcal{G}_{0})$ and  $(\mathcal{G}_{1})$ on the nonlinear terms, one shows that the solution to $(\ref{7})$ exists and is defined not only in a sufficiently small neighborhood of the secant  $\Gamma$, but also in a  neighborhood of the entire set $S$ in $\mathcal{Y}$.\\  We make the following hypotheses: let $p$ $\in$ $\mathbb{N}\cup \lbrace+\infty \rbrace$.\\
$\bullet$ $(\alpha_{p})$:
$\rightarrow$ The functions $ A^{\lambda \mu}$ are defined and of class $\mathcal{C}^{2p+1}$ in a domain $\mathbb{R}^{2}\times B \times W $ where $B$ is a bounded closed of $\mathbb{R}^{2}$ and $W$ is an open of $\mathbb{R}^{n}$;\\
$\rightarrow$ In $\mathbb{R}^{2}\times B \times W $, $A^{\lambda \mu}(x^{\alpha},u_{s}(x^{\alpha}))$ defines a quadratic form
 signature +~-~-~-, $A^{00}>0$,
 $A^{ij}X_{i}X_{j}$ is negative definite $(i,j = 1,2,3)$;\\
$\rightarrow$ There exists $(a_{r})$ $\in$ $W$ such that $A^{\lambda\mu}(x^{\alpha}_{0},a_{r})=\eta^{\lambda\mu}$ (Minkowski metric).\\
$\bullet$ $(\beta_{p})$: The $f_{r}(x^{\alpha},u_{s},u_{s \nu})$ are functions of class $\mathcal{C}^{2p-1}$ in a domain $\mathbb{R}^{2}\times B \times W \times Z$ where $Z$ is an open $\mathbb{R}^{4n}$.\\$\bullet$ $(\gamma_{p})$: The $\bar{u}_{s}$ are maps of $\mathbb{R}^{2}\times B $ in $W$ of class $C^{p+2}$ on each of the hypersurfaces $S^{w}$, continuous on the secant $\Gamma$ and such that $\bar{u}_{s}(x^{\alpha}_{0})=a_{s}$, the point $a_{s\lambda}=(D_{\lambda}\bar{u}_{s}(x^{\alpha}_{0}))$ $\in$ $Z$.\\
 $\bullet$ $\mathcal{(C)}$: The $S^{w} (w=1,2)$ are two characteristic hypersurfaces for the differential operator $\bar{L}=\bar{A}^{\lambda \mu}\partial^{2}_{\lambda \mu}$, with $\bar{A}^{\lambda \mu}(x^{\alpha})=A^{\lambda \mu}(x^{\alpha},\bar{u}_{s}(x^{\alpha}))$, and secants along the 2-surface $\Gamma$~ie:
$(x^{i})$ $\in$ $D^{w}$
\begin{equation}\label{tsv}
A^{00}((-1)^{w-1}x^{1},x^{i},\varphi^{w}_{s}(x^{i}))+2(-1)^{w}A^{01}((-1)^{w-1}x^{1},x^{i},
\varphi^{w}_{s}(x^{i}))+A^{11}((-1)^{w-1}x^{1},x^{i},\varphi^{w}_{s}(x^{i}))=0,
\end{equation}
$\bullet$ $(\lambda_{p})$: $\forall$ $r$, $\forall$ $w$,  $\varphi^{w}_{r} \in C^{2p+1}(D^{w})$, continuous on  $S^{w}$, $\varphi_{r}^{1}=\varphi_{r}^{2}$ on $\Gamma$ and $\varphi^{w}_{r}(x^{\alpha}_{0})=a_{r}$.\\
$\bullet$ \textbf{Geometric hypothesis} $(\mathcal{G})$: For all
$\sigma$ $\in$ $\mathbb{R}^{*}_{+}$, $\mathcal{P}_{(\sigma)}$ is
space for
$\bar{L}$.\\
$\bullet$ \textbf{Structure hypothesis} $(\mathcal{G}_{0})$:
 The $A^{\lambda\mu}$ are class functions $\mathcal{C}^{\infty}$ in $\mathbb{R}^{2}\times B\times W$ linear with respect to the $u_{s}$.\\
$\bullet$ \textbf{Structure hypothesis $(\mathcal{G}_{1})$}:
 The $f_{r}$ are functions of class $\mathcal{C}^{\infty}$ in $\mathbb{R}^{2}\times B \times W\times Z$
 and such that for all $r$ $\in$ $[1,n]\bigcap\mathbb{N}$, for all $w$ $\in$ $\{1;2\}$,
  $f_{r}(x^{\alpha},[u_{s}]^{w},[\partial_{\nu}u_{s}]^{w})$ is linear with respect to $[\partial_{0}u_{s}]^{w}$ when we replace
  $(x^{\alpha})$ by $((-1)^{w-1}x^{1},x^{1},x^{2},x^{3})$ and the $[\partial_{i}u_{s}]^{w}$ by their value
  $\partial_{i}[u_{s}]^{w}-(-1)^{w-1}\delta_{i1}[\partial_{0}u_{s}]^{w}$ , for
  $i=1;2;3$.\\
$\bullet$ \textbf{Structure hypothesis $(\mathcal{G}_{2})$}:
 The $f_{r}$ are functions of class $\mathcal{C}^{\infty}$
 in $\mathbb{R}^{2}\times B \times W\times Z$ linear with respect
 aux $\partial_{\nu} u_{s}$, where $Z$ is an open of
 $\mathbb{R}^{4n}$.

\begin{remarq}The geometric hypothesis $(\mathcal{G})$ allows the interior of the domain $\mathcal{Y}$ to be layered by spatial hypersurfaces so as to give rise to a spatio-characteristic hyperbolic hyper-quasilinear problem of the second order.
\end{remarq}
  \begin{remarq}
  (i) The structure hypothesis $(\mathcal{G}_{2})$ is a
  sufficient condition for the structural hypothesis
  $(\mathcal{G}_{1})$ is verified.\\
  (2i) The structure hypothesis $(\mathcal{G}_{1})$ leads to the restriction to $S^{w}$ of the equations $(G_{r})$ being reduced
  to the first order linear partial differential equations of unknowns $[\partial_{0}u_{s}]^{w}$ which can be solved globally in $S^{w}$.
 \\ (3i) If the functions $f_{r}$ are of the form $L_{r}+Q_{r}$ with the $L_{r}$ verifying the hypothesis $(\mathcal{G}_{2})$ and the $Q_{r}$
  being quadratic class functions $C^{\infty}$ in $\partial_{s}u_{s}$ and satisfying the Klainerman nullity condition
    \cite{lr} then the $f_{r}$ verify the structure hypothesis
  $(\mathcal{G}_{1})$
 \end{remarq}
Note: for $\alpha=(\alpha_{i})$ $\in$ $\mathbb{N}^{3}$, $\partial^{\alpha}=\frac{\partial^{\mid\alpha\mid}}{(\partial x^{1})^{\alpha_{1}} (\partial x^{2})^{\alpha_{2}}(\partial x^{3})^{\alpha_{3}}}$; for $\alpha=(\alpha_{\nu})$ $\in$ $\mathbb{N}^{4}$, $D^{\alpha}=\frac{\partial^{\mid\alpha\mid}}{(\partial x^{0})^{\alpha_{0}}(\partial x^{1})^{\alpha_{1}} (\partial x^{3})^{\alpha_{3}}}$ and for $k$ $\in$ $\mathbb{N}$, $\partial^{k}_{0}=\frac{\partial^{k}}{(\partial x^{0})^{k} }$.\\$\bullet$ For any vector function $v=(v_{r})$ defined in a domain of $S^{w}$ and for all $p$ $\in$ $\mathbb{N}$ we set:\\
$\parallel v\parallel_{ H^{p}(\sum_{t}^{w},S^{w})}=(\sum\limits_{r}\sum\limits_{\mid \alpha\mid \leq p} \int_{\sum_{t}^{w}} \mid \partial^{\alpha} v_{r} \mid ^{2} d\sigma (\sum_{t}^{w}) )^{\frac{1}{2}}$, where
$d\sigma (\sum_{t}^{w})$ is the measure induced on $\sum_{t}^{w}$ by $dx'=dx^{1} dx^{2} dx^{3}$;\\
$\parallel v\parallel_{ H^{p}(S^{w}_{t})}=(\sum\limits_{r}\sum\limits_{\mid \alpha\mid \leq p} \int_{S_{t}^{w}} \mid \partial^{\alpha} v_{r} \mid ^{2} dx' )^{\frac{1}{2}}$ ; \\
$\parallel v\parallel_{ E^{p}(S^{w}_{t})}=ess\sup \limits_{ \tau \in ]0, t]}\parallel v\parallel_{ H^{p}(\sum_{\tau}^{w},S^{w})} $;\\
if the second members exist and the derivatives being taken in the sense of distributions.\\
$\bullet$ For any vector function $v=(v_{r})$ defined on $\mathcal{Y}_{t}$, we set: $\forall$ $p$ $\in$ $\mathbb{N}$,\\
$\parallel v\parallel_{ H^{p}(G_{\tau},\mathcal{Y})}=(\sum\limits_{r}\sum\limits_{\mid \alpha\mid \leq p} \int_{G_{\tau}} \mid  D^{\alpha} v_{r} \mid ^{2} dx' )^{\frac{1}{2}}$;\\
$\parallel v\parallel_{ K^{p}(\mathcal{Y}_{t})}=(\int_{0}^{t}\tau^{-1} \parallel v\parallel^{2}_{ H^{p}(G_{\tau},\mathcal{Y})} d\tau  )^{\frac{1}{2}}$ ;
$\parallel v\parallel_{ E^{p}(\mathcal{Y}_{t})}=ess\sup \limits_{ \tau \in ]0, t]} \tau^{-\frac{1}{2}} \parallel v\parallel_{ H^{p}(G_{\tau},\mathcal{Y})} $;\\
$\parallel v\parallel_{ \mathcal{K}^{p}(\Gamma,S^{w},\mathcal{Y})}=(\sum\limits_{k=0}^{p-1}  \parallel [\partial^{k}_{0}v] \parallel^{2}_{ H^{2(p-k)-1}(\Gamma,S^{w})}   )^{\frac{1}{2}}$; $\parallel v\parallel_{ \mathcal{K}^{p}_{1}(\Gamma,S^{w},\mathcal{Y})}=(\sum\limits_{k=1}^{p-1}  \parallel [\partial^{k}_{0}v] \parallel^{2}_{ H^{2(p-k)-1}(\Gamma,S^{w})}   )^{\frac{1}{2}}$;
\\
$\parallel v\parallel_{ \mathcal{K}^{p}_{1}(S_{t}^{w},\mathcal{Y})}=(\sum\limits_{k=1}^{p-1}  \parallel [\partial^{k}_{0}v] \parallel^{2}_{ H^{2(p-k)-1}(S_{t}^{w})}   )^{\frac{1}{2}}$ ; $\parallel v\parallel_{ \mathcal{E}^{p}_{1}(S_{t}^{w},\mathcal{Y})}=(\sum\limits_{k=1}^{p-1}  \parallel [\partial^{k}_{0}v]^{w} \parallel^{2}_{ E^{2(p-k)-1}(S_{t}^{w})}   )^{\frac{1}{2}}$;
\\
$\parallel v\parallel_{ \mathcal{K}^{p}(S_{t}^{w},\mathcal{Y})}=(\sum\limits_{k=0}^{p-1}  \parallel [\partial^{k}_{0}v] \parallel^{2}_{ H^{2(p-k)-1}(S_{t}^{w})}   )^{\frac{1}{2}}$ ; $\parallel v\parallel_{ \mathcal{E}^{p}(S_{t}^{w},\mathcal{Y})}=(\sum\limits_{k=0}^{p-1}  \parallel [\partial^{k}_{0}v]^{w} \parallel^{2}_{ E^{2(p-k)-1}(S_{t}^{w})}   )^{\frac{1}{2}}$;
\\
$\parallel v\parallel_{ \mathcal{K}^{p}(\mathcal{Y}_{t})}=( \parallel v \parallel^{2}_{ K^{p}(\mathcal{Y}_{t})}+ \sum\limits_{w=1}^{2}  \parallel v \parallel^{2}_{ \mathcal{K}^{p}(S_{t}^{w}, \mathcal{Y})}   )^{\frac{1}{2}}$ ; $\parallel v\parallel_{ \mathcal{E}^{p}(\mathcal{Y}_{t})}=( \parallel v \parallel^{2}_{ E^{p}(\mathcal{Y}_{t})}+  \sum\limits_{w=1}^{2} \parallel v \parallel^{2}_{ \mathcal{E}^{p}(S_{t}^{w}, \mathcal{Y})}   )^{\frac{1}{2}}$ ;\\
$\parallel v\parallel_{ \mathcal{K}^{p}_{1}(\mathcal{Y}_{t})}=( \parallel v \parallel^{2}_{ K^{p}(\mathcal{Y}_{t})}+ \sum\limits_{w=1}^{2}  \parallel v \parallel^{2}_{ \mathcal{K}^{p}_{1}(S_{t}^{w}, \mathcal{Y})}   )^{\frac{1}{2}}$ ; $\parallel v\parallel_{ \mathcal{E}^{p}_{1}(\mathcal{Y}_{t})}=( \parallel v \parallel^{2}_{ E^{p}(\mathcal{Y}_{t})}+  \sum\limits_{w=1}^{2} \parallel v \parallel^{2}_{ \mathcal{E}^{p}_{1}(S_{t}^{w}, \mathcal{Y})}   )^{\frac{1}{2}}$ ;\\
if the second members exist and the derivatives being taken in the sense of distributions.\\
We consider the following functional spaces:\\
$\bullet$ $C^{\infty}(\mathcal{Y}_{t,\sigma})$ is the space of restrictions to $\mathcal{Y}_{t,\sigma}$ of functions of class $C^{\infty}$ on $\mathbb{R}^{N}$.\\
$\bullet$ $K^{p}(\mathcal{Y}_{t,\sigma})$ (resp. $E^{p}(\mathcal{Y}_{t,\sigma})$) is the subspace of $H^{p}(\mathcal{Y}_{t,\sigma})$ formed of vector functions $v=(v_{r})$ for which $\parallel v\parallel_{ K^{p}(\mathcal{Y}_{t,\sigma})} < +\infty$ (resp. $\parallel v\parallel_{ E^{p}(\mathcal{Y}_{t,\sigma})} < +\infty$).\\
$\bullet$ $E^{p}(S_{t}^{w})$ is the subspace of $H^{p}(S_{t}^{w})$ formed of functions
vectors $v=(v_{r})$ for which $\parallel v\parallel_{ K^{p}(S_{t}^{w})} < +\infty$.\\$\bullet$ $\mathcal{K}^{p}(\mathcal{Y}_{t,\sigma})$ is
 the subspace of $K^{p}(\mathcal{Y}_{t,\sigma})$ formed of vector functions $v=(v_{r})$ for which
 $\parallel v\parallel_{ \mathcal{K}^{p}(\mathcal{Y}_{t,\sigma})} < +\infty$.\\
 $\bullet$ $\mathcal{E}^{p}(\mathcal{Y}_{t,\sigma})$ is
 the subspace of $E^{p}(\mathcal{Y}_{t,\sigma})$ formed of vector functions $v=(v_{r})$ for which
 $\parallel v\parallel_{ \mathcal{E}^{p}(\mathcal{Y}_{t,\sigma})} < +\infty$.\\
$\bullet$ $\hat{\mathcal{E}}^{p}(\mathcal{Y}_{t,\sigma})$ is the
closing
in $\mathcal{E}^{p}(\mathcal{Y}_{t,\sigma})$ of $C^{\infty}(\mathcal{Y}_{t,\sigma})$.\\
$\bullet$ $\hat{E}^{p}(S^{w}_{t})$ is the closure in
$E^{p}(S^{w}_{t})$ of the space of
 restrictions to $S^{w}_{t}$ of functions of class $C^{\infty}$ on $\mathbb{R}^{4}$.\\
$\bullet$ $\hat{E}^{p}_{loc}(S^{w})$ is the space of vector functions $v$ defined in $S^{w}$ and such that $\forall$ $t$ $\in$ $\mathbb{R}^{*}_{+}$, $v\mid_{S^{w}_{t}}$ $\in$ $\hat{E}^{p}(S^{w}_{t})$.\\
$\bullet$ $\hat{\mathcal{E}}^{p}_{loc}(\mathcal{Y}^{(g)})$ is the space of vector functions $v$ defined in $\mathcal{Y}^{(g)}$ and such that $\forall$ $t$ $\in$ $\mathbb{R}^{*}_{+}$, $v\mid_{\mathcal{Y}_{t,g(t)}}$ $\in$ $\hat{\mathcal{E}}^{p}(\mathcal{Y}_{t,g(t)})$.
\section{ Formulation and proofs of the results}
 \begin{defn}\cite{3,4} A part $\mathcal{Y}$ of $\mathbb{R}^{4}$ whose boundary contains $S$ is said to be causal if: $\forall$ $M_{0}$ $\in$ $ \mathcal{Y}$, $(C_{M_{0}}^{-})\subseteq \mathcal{Y}$ and $M_{0}$ is the unique singular point of $(C_{M_{0}}^{-})$; where $C_{M_{0}}^{-}$ is the half-cono$\ddot{i}$ of characteristic for $\bar{L}$ coming from $M_{0}$ and directed towards the past; $(C_{M_{0}}^{-})$ being the part of $C_{M_{0}}^{-}$ located between $M_{0}$ and $S$; $S_{0}(M_{0})$ denotes the 2-surface $C_{M_{0}}^{-}\cap S$.
\end{defn}

\begin{theorem}\label{er} 1)Let $s$ $\in$ $\mathbb{N}$, $s\geq6$. If the $A^{\lambda\mu}$
verify the structure hypothesis $(\mathcal{G}_{0})$, the $f_{r}$
verify the structure hypothesis $(\mathcal{G}_{1})$,
if the hypotheses $(\alpha_{\infty})$, $(\beta_{\infty})$, $(\gamma_{\infty})$, $\mathcal{(C)}$ and $\mathcal{(G)}$ are verified, and if the initial data $\varphi^{w}=(\varphi^{w}_{r})$ are such that:\\
(i) $\varphi^{w}$ $\in$ $\hat{E}^{2s-1}_{ loc}(S^{w})$\\
(ii) $\varphi^{1}=\varphi^{2}$ on $\Gamma$;\\
Then there exists a map $f$ of $\mathbb{R}^{*}_{+}$ $\rightarrow$ $\mathbb{R}^{*}_{+}$, and a unique solution $u=(u_{r})$ of
 hyper-quasilinear Goursat problem $(\ref{7})$ in the domain $\mathcal{Y}^{(f)}$ and such that $u$ $\in$ $\hat{\mathcal{E}}^{s}_{loc}(\mathcal{Y} ^{(g)})$.\\
2) If besides 1) we assume that the initial data $\varphi^{w}=(\varphi^{w}_{r})$ are
 restrictions to $S^{w}$ of functions of class $C^{\infty}$ on $\mathbb{R}^{2}\times B$,
  then there exists a map $f$ of $\mathbb{R}^{*}_{+}$ $\rightarrow$ $\mathbb{R}^{*}_{+}$, and a solution $u=(u_{r})$
   of the hyper-quasilinear Goursat problem $(\ref{7})$ defined in the domain $\mathcal{Y}^{(f)}$ and such
   that $u$ $\in$ $C^{\infty}(\mathcal{Y}^{(f)})$.

\end{theorem}

   \begin{theorem}\label{rtt} Under the hypotheses $(\alpha_{4})$, $(\beta_{4})$,
   $(\gamma_{4})$, $(\lambda_{4})$, $(\mathcal{G}_{0})$ and $(\mathcal{G}_{1})$ we have:\\
1)- For any solution $u=(u_{r})$ of $(\ref{7})$, six times
differentiable and admitting derivatives up to order five
continuous and bounded, the functions $u_{r}$ as well as their derivatives
up to order four verify the following integral system: for
any point $M({x_{0}^{\alpha}})$ $\in$ $\mathcal{Y}$

 \begin{center}

 $(\mathcal{G}_{S}): \left\{
               \begin{array}{ll}
            4 \pi U_{S}(x_{0}^{\alpha})=\int_{0}^{2 \pi}d\lambda_{3}\int_{0}^{ \pi}d\lambda_{2}\sin \lambda_{2}\int_{0}^{\psi(x_{0}^{\alpha},\lambda_{h})}d\lambda_{1}  \mathcal{H}_{S}(U_{T},\Omega^{R}_{T},\hat{\Omega}^{R}_{T})& \hbox{} \\

~~~~~~~~~~~~~+ \int_{0}^{2 \pi}d\lambda_{3}\int_{0}^{ \pi}d\lambda_{2}\mathcal{J}_{S}([U_{T}],\Omega^{R}_{T})& \hbox{}\\
\Omega^{R}_{S}(x^{\alpha}_{0};\lambda_{h})=\Omega^{R}_{0S}+\int^{\lambda_{1}}_{0}H_{T}^{R}(U_{Q})\Omega^{T}_{S}d\lambda_{1} & \hbox{}
                \end{array}
              \right.
 $
 \end{center}
 where
  \begin{equation}\label{201s}
   \left\{
               \begin{array}{ll}
         U_{R}=(u_{r},u_{r \alpha}, u_{r \alpha \beta}, u_{r \alpha \beta \gamma},u_{r \alpha \beta \gamma \delta }),   & \hbox{} \\
         \mathcal{J}_{S}([U_{T}],\Omega^{R}_{T})= E^{i}_{S}([U_{T}],\Omega^{R}_{S})(\triangle^{1}_{i}-\frac{\partial \psi}{\partial \lambda_{h}}\triangle^{h}_{i}),& \hbox{} \\
         L_{S}^{R}=\frac{\partial^{2}}{\partial y^{i}\partial y^{j}}((\stackrel{ \text{\textasteriskcentered} }{[A^{ijT}_{1}]^{w}}\stackrel{ \text{\textasteriskcentered} }{\Phi^{w}_{T}}+\stackrel{ \text{\textasteriskcentered} }{[A^{ij}_{2}]^{w}})\sigma_{S}^{R})-\frac{\partial }{\partial y^{i}}(\stackrel{ \text{\textasteriskcentered} }{[B^{iR}_{T}]}\sigma_{S}^{T}),& \hbox{} \\
          \mathcal{H}_{S}(U_{T},\Omega^{R}_{T},\hat{\Omega}^{R}_{T}) =\square([U_{R}]L^{R}_{S}(U_{T},\Omega^{R}_{T},\hat{\Omega}^{R}_{T})+\sigma w^{R}_{S}  [f_{R}(U_{T})]),  & \hbox{}\\
           \Omega^{R}_{S}=(\omega^{R}_{S},~\omega^{R}_{S,i},\omega^{R}_{S,ij}),~ \hat{ \Omega}^{R}_{S}=(\frac{w^{R}_{S}-\delta^{R}_{S}}{\lambda_{1}};\frac{w^{R}_{S,i}}{\lambda_{1}};\frac{w^{R}_{S,ij}}{\lambda_{1}}),
           & \hbox{}\\  \Omega^{R}_{0S}=(\delta^{R}_{S},0,0),~~ w^{R}_{S,i}=\frac{\partial w^{R}_{S}}{\partial p^{0}_{i}}, ~ w^{R}_{S,ij}=\frac{\partial^{2} w^{R}_{S}}{\partial p^{0}_{i}\partial p^{0}_{j}},
               & \hbox{}\\ \bigtriangleup=\mid  \frac{D(x^{i})}{D(\lambda_{j})}\mid, ~\square=\frac{\bigtriangleup}{\sin \lambda_{2}}, ~\bigtriangleup^{i}_{j}=~minor~associated~with~\bigtriangleup, \sigma=-\mid \square \mid^{-\frac{1}{2}},  & \hbox{}\\p^{0}_{1}=\sin \lambda_{2}\cos \lambda_{3},~ p^{0}_{2}=\sin \lambda_{2}\sin \lambda_{3}, ~p^{0}_{3}=\cos \lambda_{2}
                & \hbox{}\\ sur ~S_{0}(M_{0})~on~a~\lambda_{1}=\psi(x^{\alpha}_{0};~\lambda_{h});~ i,j=1,2,3; ~\alpha,\beta,\gamma=0,...,3. & \hbox{}\\
       H_{T}^{R}(U_{Q}) ~is~a~triangular~matrix~of~the~functions~A^{\lambda\mu s}_{1},A^{\lambda\mu}_{2},f_{r}~and & \hbox{}\\of ~their~ derivatives~ up~ to~ order~ three.  & \hbox{}
                \end{array}
              \right.
\end{equation}The $w^{R}_{S}$, $E^{i}_{S}$ are the auxiliary functions introduced in \cite{3, 4}; the $ p^{0}_{i}$ are the parameters making it possible to identify the bicharacteristics coming from $M_{0}$, which generate $C_{M_{0}}^{-}$ , $\lambda_{1}$ is a parameter making it possible to identify the points of a given bicharacteristic of $C_{M_{0}}^{-}$ .\\
2)-There exists a positive constant $l$ and positive constants $C(l)$, $M(l)$, $B(l)$ such that for any causal domain $\mathcal{Y}_{0}$ whose boundary contains $S$ and which is contained in the domain $\mathcal{D}(C(l)$, $M(l)$, $B(l))$ defined by:
\begin{center}
$\left\{  (x^{\alpha}) \in Y : \left\{
               \begin{array}{ll}
         0<a < M(l)  & \hbox{} \\

  a\leq C(l)\mid x^{1}\mid
                \end{array}
              \right.  or~~ \left\{
               \begin{array}{ll}
         0<a < M(l)  & \hbox{} \\

     a >C(l)\mid x^{1}\mid  \hbox{} \\ x^{0}<B(l)
                \end{array}
              \right.  \right\}$
\end{center}with $a=x^{0}-\mid x^{1}\mid $, the integral system $(\mathcal{G}_{S})$ admits a unique solution $(U_{S},\Omega^{R}_{S})$ in the space of continuous and bounded functions and we have:\\
 $(i)$ $\forall$ $(x^{\alpha})$ $\in$ $\mathcal{Y}_{0}$, $\mid U_{S}(x^{0},x^{i})-\Phi_{S}(x^{0},x^{i}) \mid \leq l$; \\
 $(ii)$ the functions $U_{S}$ take on $S^{w}$ the values $\Phi_{S}=( [u_{s}]^{w},[u_{s \alpha}]^{w},[u_{s \alpha\beta}]^{w},[u_{s \alpha\beta \gamma}]^{w},[u_{s \alpha\beta \gamma\delta}]^{w})$.\\
\end{theorem}
\textbf{Idea of the proof of Theorem $\ref{rtt}$}\\
1. Formally writing, for problem $(\ref{7})$, the Kirchhoff formulas as in the case of semilinear systems, we note that although the functions $A^{\lambda\mu}$
 depend linearly on the unknowns $u_{s}$
	?, the resulting identities do not define an integral system. Therefore, we reduce the study of problem $(\ref{7})$ to the following hyper-quasilinear Goursat problem:
\begin{equation}\label{8f}
\left\{
                \begin{array}{ll}
                  (H_{r}):  A^{\lambda \mu m}_{1}(x^{\alpha})u_{m}\partial^{2}_{\lambda\mu}u_{r}+
                   A_{2}^{\lambda \mu}(x^{\alpha})\partial^{2}_{\lambda\mu}u_{r}
                   +f_{r}(x^{\alpha},u_{s}, \partial_{\nu}u_{s})=0 ~~in ~~~~\mathcal{Y}&  \hbox{} \\ & \hbox{} \\~~~~~~~~~~~~~

~~~~~~u_{r}=\bar{u}_{r}=\varphi_{r}^{w} ~~on ~~S^{w};\end{array}
              \right.
\end{equation}
Writing formally, for problem $(\ref{8f})$, the Kirchhoff formulas as in the case of semilinear systems, we observe that although the functions $f_{r}$ depend linearly on the first derivatives of the unknowns $u_{s}$?, the resulting identities do not define an integral system either.
It is worth noting that if the functions $A^{\lambda\mu}$ are linear with respect to the unknowns $u_{s}$, and if the auxiliary functions appearing in the Kirchhoff formulas do not depend on the derivatives of the functions $u_{r}$, then the formally written Kirchhoff relations define a nonlinear integral system. In this case, it becomes possible to apply results concerning semilinear problems to the hyper-quasilinear Goursat problem $(\ref{8f})$, provided that the auxiliary functions do not depend on the derivatives of the unknown functions $u_{r}$.
To ensure that these auxiliary functions are independent of the derivatives of $u_{s}$, one is led to differentiate the equations $(H_{r})$ four times and to apply to the resulting system obtained by adjoining the equations $(H_{r})$ to those derived after four differentiations arguments analogous to those established by D. Houpa and M. Dossa \cite{4} for semilinear systems.
\\2.  The restrictions to S of the derivatives up to order four of any possible solution of the hyper-quasilinear problem $(\ref{8f})$, denoted by $(\Phi_{S})$, are uniquely determined on the whole of $S$. The proof of 2) is then similar to that given in \cite{3, 4} in the case of the semilinear Goursat problem.\\
3. The solution of the integral system $(G_{S})$
 associated with problem $(\ref{8f})$ is constructed in the space of bounded continuous functions defined on a causal domain $\mathcal{Y}_{1}$ a neighbourhood of $S$ in $\mathcal{Y}$ (endowed with the metric of uniform convergence), as a fixed point of a contraction mapping $\Theta$ sending a ball centered at $(\Phi_{S})$ into itself.
The proof of 3) is then similar to that given in \cite{3, 4} in the case of the semilinear Goursat problem.\par
\textbf{Idea of the proof of the theorem $\ref{er}$}.\\
 To establish this result, it suffices to show that $\forall$ $T \in \mathbb{R}^{*}_{+}$, there exists $f(T)$ $\in$ $\mathbb{R}^{*}_{+}$
 such that the problem $(\ref{7})$ admits in $\mathcal{Y}_{T,f(T)}$ a unique solution $u$ $\in$ $\hat{\mathcal{E}}^{s}(\mathcal{Y} ^{(f)})$.\\
For reasons of density and completeness of the Sobolev type spaces considered, we can assume all the initial data $C^{\infty}$.
By theorem 3.2, for all $T$ $\in$ $\mathbb{R}^{*}_{+}$,
it exists, given the geometric hypothesis $(\mathcal{G})$, $f(T)$ $\in$
$\mathbb{R}^{*}_{+}$ such that
 integral system $(\mathcal{G}_{S})$ of Kirchhoff formulas associated with the hyper-quasilinear systems of Goursat $(\ref{7})$ admits
 a unique solution $(v_{r},v_{r\alpha},v_{r\alpha \beta}, v_{r\alpha \beta \gamma},v_{r\alpha \beta \gamma\theta})$ in
  the space of continuous and bounded functions in $\mathcal{Y}_{T,f(T)}$.\\
According to \cite{3}, there exists $T_{0}$ $\in$ $]0,T]$ such that in the
domain $\mathcal{Y}_{T_{0},f(T)}$ the problem $(\ref{7})$ admits a
unique solution $u=(u_{r})$ $\in$ $C^{\infty
}(\mathcal{Y}_{T_{0},f(T)})$. We deduce that the functions $u_{r}$ and
its derivatives up to order four verify the integral system
$(\mathcal{G}_{S})$. By uniqueness of the solution of this system
integral we therefore have: $v_{r}=u_{r}$, $
v_{r\alpha}=\partial_{\alpha}u_{r}$,
$v_{r\alpha\beta}=\partial^{2}_{\alpha\beta}u_{r}$, $
v_{r\alpha\beta\gamma}=\partial^{3}_{\alpha\beta\gamma}u_{r}$,
$v_{r\alpha\beta\gamma
\theta}=\partial^{4}_{\alpha\beta\gamma\theta}u_{r}$.\\ We deduce that
$\forall$ $T_{1}$ $\in$ $]0,T]$, for every solution  $u$ of $(\ref{7})$ defined in
 $\mathcal{Y}_{T_{1},f(T)}$ satisfies the a priori estimate:
$\parallel u\parallel_{\mathcal{E}^{s}(\mathcal{Y}_{T_{1},f(T)})}\leq
K$ where $K$ is a strictly positive constant which only depends on $T$,
$\sum\limits_{w=1}^{2}\parallel \varphi^{w}
\parallel_{E^{2s-1}(S^{w}_{T})}$ and the constant
$K_{0}=\max\limits_{r,\alpha,\beta,\gamma}\max\limits_{(x^{\alpha})
\in \mathcal{Y}_{T_{1},f(T)}}\lbrace \mid v_{r}(x^{\alpha}) \mid,
\mid v_{r\alpha}(x^{\alpha}) \mid, \mid v_{r\alpha\beta}(x^{\alpha})
\mid, \mid v_{r\alpha\beta\gamma}(x^{\alpha}) \mid,\mid
v_{r\alpha\beta\gamma\theta}(x^{\alpha}) \mid  \rbrace$ because $s\geq
6> \frac{3}{2}+4$.
It follows, via the local resolution of a mixed spacelike characteristic Cauchy problem arising in the foliation of the domain by spacelike hypersurfaces (cf. \cite{5}), and using the results of \cite{2,4,47}, that problem $(\ref{7})$ admits a unique solution in the whole domain  $Y_{T,f(T)}$
, and therefore in the entire domain $\mathcal{Y}^{(f)}=\cup_{T \in \mathbb{R}^{*}_{+}}\mathcal{Y}_{T,f(T)}$.

\section{Applications to Einstein's vacuum equations in harmonic gauge } The purpose of this section is to apply Theorem $\ref{er}$, established in the previous paragraph, to the vacuum Einstein equations in an unknown spacetime $\mathcal{M}$.
  Unfortunately this application is not carried out without difficulties. In fact, we
considers in a space-time $\mathcal{M}$, a system of relativistic field equations or gauge theories, with initial conditions carried by a
$S$ spatial or characteristic hypersurface. The Cauchy problem thus obtained
is, in general, ill-posed because the field equations expressed in a system
of any coordinates are not, in general, a system of evolution (hyperbolic or parabolic). To resolve this difficulty, we have to add
to the system of equations of the fields studied an additional condition called \textbf{gauge condition} which, taking into account the deep structure of the equations
studied, must have the following properties:
  \par $\bullet$ when this gauge condition is verified everywhere in space-time, the
system of field equations studied is reduced to a system of evolution.
 \par $\bullet$ when the associated evolution system is verified everywhere and the condition
gauge is verified on the initial hypersurface $S$, then this condition is
checked everywhere. \par It follows that when a gauge condition is chosen, the Cauchy problem
for the field equations is split into two sub-problems: the problem of
initial constraints and the problem of evolution.\\
The problem of initial constraints consists of constructing, from the choice
arbitrary on $S$ of certain components of the fields called \textbf{independent data}, complete initial data such as the solution of the problem of
the evolution associated with these initial data verifies the gauge condition on $S$.
The problem of evolution consists of resolving the system in space-time
of evolution (also called reduced equations) obtained for initial data
complete, solution of the problem of initial constraints.
We therefore consider a space time $(\mathcal{M}, g)$
where $\mathcal{M}$ is a differentiable manifold of dimension 4,
separated, infinitely countable, of class $C^{\infty}$ and $g$ a
unknown Lorentzian metric i.e. hyperbolic Riemannian
normal, twice covariant symmetric tensor of signature +~-~-~-
and such that the geometric hypothesis $(\mathcal{G})$ is verified.
We consider $L=g^{\lambda\mu}\partial_{\lambda\mu}^{2}$ a $x^{0}$-hyperbolic differential operator of the second order, $S^{w}(w=0,1)$ of equations $S^{w}=\{ (x^{\alpha})\in
\mathcal{M}: x^{0}=(-1)^{w+1}x^{1}\}$ two hypersurfaces
characteristics for the second order differential operator
$L$ and secants
following the 2-surface $\Gamma$ of $\mathcal{M}$ defined by
$x^{0}=0=x^{1}$. We suppose
furthermore that $\forall \sigma \in \mathbb{R}^{*}_{+}$, the hypersurface
$\mathcal{P}_{\sigma}$ defined by:
$(x^{0})^{2}-(x^{1})^{2}=\sigma^{2}$ is spatial for $L$
that is, the geometric hypothesis $(\mathcal{G})$ is
verified. We recall that the geometric frame and the frame
functional are the same as those in the previous section. So we
has: $\mathcal{Y}=\{ (x^{\alpha})\in \mathcal{M}: x^{0}\geq \mid
x^{1}\mid\}$, $\mathcal{Y}_{T}=\mathcal{Y}\bigcap\{(x^{\alpha})\in
\mathcal{M}: x^{0}\leq T\} $, $S^{w}_{T}=S^{w}\bigcap
\mathcal{Y}_{T}$, $\mathcal{Y}_{(\sigma)}=\lbrace (x^{\alpha}) \in
\mathcal{M}:
 \mid x^{1}\mid \leq x^{0} \leq \sqrt{(x^{1})^{2}+\sigma^{2}} \rbrace$,
   $\mathcal{Y}_{T,\sigma}=\mathcal{Y}_{T}\cap
   \mathcal{Y}_{(\sigma)}$,
 $\mathcal{Y}^{(f)}=\cup_{T \in \mathbb{R}^{*}_{+}} \mathcal{Y}_{T,f(T)}$, avec $T,\sigma \in \mathbb{R}^{*}_{+}$,
 $f$ est une application de $\mathbb{R}^{*}_{+}$ dans $ \mathbb{R}^{*}_{+}$.
 Let $x=(x^{\alpha})=(x^{0},x^{1},x^{a})$,
$\alpha=0,1,2,3$; $a=2.3$; the global coordinate system of
$\mathcal{M}$. In the coordinate system $(x^{\alpha})$, the
Einstein's vacuum equations are written in the following form:
 \begin{equation}\label{800f}
   R_{\alpha\beta}=0\end{equation}
where $\bullet$ The $ R_{\alpha\beta}$ designate the components of the
Ricci tensor of the metric $g$ defined by:
\begin{equation}\label{800}
R_{\alpha\beta}(g)\equiv\Gamma^{\lambda}_{\alpha\beta,\lambda}-\Gamma^{\lambda}_{\alpha\lambda,\theta}+
\Gamma^{\lambda}_{\lambda\theta}\Gamma^{\theta}_{\alpha\beta}-\Gamma^{\lambda}_{\beta\theta}\Gamma^{\theta}_{\alpha\lambda}
\end{equation}
 and who
decompose (cf. \cite{ty}) in the coordinate system
$(x^{\alpha})$ in the following form:\begin{equation}\label{803}
R_{\alpha\beta}(g)= \tilde{R}_{\alpha\beta}(g)
+\frac{1}{2}(g_{\lambda\alpha}\Gamma^{\lambda}_{,\beta}+g_{\lambda\beta}\Gamma^{\lambda}_{,\alpha})
\end{equation}with:
\begin{equation}\label{ty}
\tilde{R}_{\alpha\beta}(g)=-\frac{1}{2}g^{\lambda\gamma}g_{\alpha\beta,\lambda\gamma}+Q_{\alpha\beta}(g,\partial
g)~~~and~~\Gamma^{\lambda}=g^{\alpha\beta}\Gamma^{\lambda}_{\alpha\beta};
\end{equation}
and

\begin{eqnarray*}
Q_{\alpha\beta}(g,\partial g)   &=&
\frac{1}{2}(g_{\delta\alpha,\beta}+g_{\delta\beta,\alpha})\Gamma^{\delta}+
\frac{1}{2}g^{\delta\eta}g^{\lambda\mu}(g_{\lambda\delta,\beta}g_{\alpha\eta,\mu}+g_{\lambda\delta,\alpha}g_{\beta\eta,\mu})
-\frac{1}{4}g^{\delta\eta}g^{\lambda\mu}g_{\delta\lambda,\alpha}g_{\mu\eta,\beta}\\
 &-&\frac{1}{2}g^{\delta\eta}g^{\lambda\mu}g_{\eta\lambda,\delta}(g_{\mu\beta,\alpha}+g_{\mu\alpha,\beta}-g_{\alpha\beta,\mu}) +
\frac{1}{4}g^{\delta\eta}g^{\lambda\mu}g_{\delta\eta,\mu}(g_{\alpha\lambda,\beta}+
g_{\beta\lambda,\alpha}-g_{\alpha\beta,\lambda}) \\
&-&
\frac{1}{2}g^{\delta\eta}g^{\lambda\mu}g_{\delta\alpha,\lambda}(g_{\mu\beta,\eta}-g_{\eta\beta,\eta})
\end{eqnarray*}
$\bullet$ The $g_{\alpha\beta}$ represent the components of the unknown metric $g$;\\
 $\bullet$ Commas ","
as indices designate the usual partial derivatives. By
example \begin{center} $\frac{\partial g_{\alpha\beta}}{\partial
x^{\theta}
}=g_{\alpha\beta,\theta}=\partial_{\theta}g_{\alpha\beta}$
\end{center}
$\bullet$ The $\Gamma^{\lambda}_{\alpha\beta}$ are the symbols of
Christoffel of the metric $g$ defined in coordinates $(x^{\alpha})$
by:\begin{equation}\label{802}
\Gamma^{\lambda}_{\alpha\beta}=\frac{1}{2}g^{\lambda\theta}(g_{\beta\theta,\alpha}
+g_{\alpha\theta,\beta}-g_{\alpha\beta,\theta});
\end{equation}
$\bullet$ The $g^{\alpha\beta}$ are the components of the matrix
inverse of that of components $g_{\alpha\beta}$.

\begin{defn}\cite{ty}
We say that the coordinate system $(x^{\alpha})$,
$\alpha=0,1,2,3$ checks the harmonic gauge condition on
$(\mathcal{M}, g)$ if:
\begin{equation}\label{kui}
\Gamma^{\alpha}=0,~~\alpha=0,1,2,3.
\end{equation}
\end{defn}By combining the relations   $(\ref{800f})$, $(\ref{803})$,
$(\ref{ty})$ and $(\ref{kui})$ we have:
\begin{equation}\label{tyr}
\tilde{R}_{\alpha\beta}(g)=-\frac{1}{2}g^{\lambda\gamma}g_{\alpha\beta,\lambda\gamma}+Q_{\alpha\beta}(g,\partial
g)=0
\end{equation}The $(g_{\alpha\beta})$ being solution of the system
 reduces $(\ref{tyr})$, passing to restrictions on $S^{w}$ and using the fact that for any function $g_{\alpha\beta}$ we have:
  \begin{equation}\label{hyr}
                           [\partial_{i} g_{\alpha\beta}]^{w}=\partial_{i} [g_{\alpha\beta}]^{w}+ (-1)^{w}\delta_{1i}[\partial_{0} g_{\alpha\beta}]^{w}
                        \end{equation} we obtain the nonlinear system
  of the first order of unknown
$[\partial_{0}g_{\alpha\beta}]^{w}$ following:
\begin{multline}\label{tyrr}
\tilde{R}_{\alpha\beta}(g)\mid_{S^{w}}=-\frac{1}{2}([g^{0i}]^{w}-(-1)^{w+1}\delta_{1j}[g^{ij}]^{w})\partial_{i}[\partial_{0}g_{\alpha\beta}]^{w}
\\+Q_{\alpha\beta}^{1 }(x,[g]^{w},\partial_{i}[g]^{w},[\partial_{0}g]^{w})+Q_{\alpha\beta}^{2}(x,[g]^{w},[\partial_{0}g]^{w})=0
\end{multline}where:
$ \mu,\nu,\alpha,\beta,\lambda,\gamma=0,1,2,3;~~i,j=1,2,3;~~w=0,1.$
\\
$\bullet$ the $Q_{\alpha\beta}^{1}$ are linear terms in $K^{w}$;\\
$\bullet$ the $Q_{\alpha\beta}^{2}$ are homogeneous quadratic forms of degree 2 in $K^{w}$.\\
Since the functions $[g]^{w}$ and $\partial_{i}[g]^{w}$ are known then the equation $(\ref{tyrr})$ can be put in the following form:
\begin{equation}\label{tyxwvt}
         -\frac{1}{2}\sum\limits_{i=1}^{3}([g^{0i}]^{w}-(-1)^{w+1}\delta_{1j}[g^{ij}]^{w})\partial_{i}K_{\alpha\beta}^{w}=F(K_{\alpha\beta}^{w}), ~~\alpha,\beta=0,...,3,j=1,...,3
       \end{equation}
      By setting $K^{w}= (K_{\alpha\beta}^{w})$ we have:
\begin{equation}\label{tyxwvt}
    -\frac{1}{2}\sum\limits_{i=1}^{3}([g^{0i}]^{w}-(-1)^{w+1}\delta_{1j}[g^{ij}]^{w}) K^{w}=F(K^{w})=L(K^{w})+Q(K^{w},K^{w})
\end{equation}with  $L(K^{w})=-Q_{\alpha\beta}^{1}$ and $Q(K^{w},K^{w})=-Q_{\alpha\beta}^{2}$.\\
Hence $(\ref{tyxwvt})$ is a first order nonlinear differential system.\\
\textbf{1. Hyperbolicity and regularity of $L= -\frac{1}{2}\sum\limits_{i=1}^{3}([g^{0i}]^{w}-(-1)^{w+1}\delta_{1j}[g^{ij}]^{w}) \partial_{i}$ }\\
To do this, we will show that its main symbol admits real eigenvalues.
Since the principal symbol corresponds to the part of the differential operator which contains the highest order derivatives, then the operator considered here is:
\begin{equation} -\frac{1}{2}\sum\limits_{i=1}^{3}([g^{0i}]^{w}-(-1)^{w+1}\delta_{1j}[g^{ij}]^{w}) \partial_{i}K^{w}_{\alpha\beta}, j=1,...,3
\end{equation}The side $F(K^{w}_{\alpha\beta})$ does not contribute to the main symbol, because it does not contain derivatives.
\\If $K^{w}$ is seen as a vector of dimension 16 (since $\alpha,\beta=0,...,3$):
$ K^{w}=(K_{00}^{w},...,K_{33}^{w}) \in \mathbb{R}^{16}$
then the main symbol $A(\xi)$ (with $\xi=(\xi_{1},\xi_{2},\xi_{3})\in \mathbb{R}^{3}$) acts on $K^{w}$ by diagonal multiplication. So it's a diagonal matrix $16\times 16$.
\begin{equation} A(\xi)=diag(-\frac{1}{2}\sum\limits_{i=1}^{3}([g^{0i}]^{w}-(-1)^{w+1}\delta_{1j}[g^{ij}]^{w}) \xi_{i})_{16 times}
  \end{equation}
The characteristic polynomial $P(\lambda,\xi)=det(A(\xi)-\lambda I)=0$.\\As $A(\xi)$ is diagonal, all these eigenvalues are:
\begin{equation}
  \lambda_{j}=-\frac{1}{2}\sum\limits_{i=1}^{3}([g^{0i}]^{w}-(-1)^{w+1}\delta_{1j}[g^{ij}]^{w}) \xi_{i}~,~j=1,...,3.
\end{equation}
Since the eigenvalues $\lambda_{j}$ are linear combinations of the components $[g^{0i}]^{w}$ and $[g^{ij}]^{w}$ with real numbers $\xi_{i}$ then the $\lambda_{j}$ are real numbers so the operator $L$ is hyperbolic.\\
Since $\lambda_{j}$ is a linear combination of $[g^{0i}]^{w}$ and $[g^{ij}]^{w}$ with real numbers $\xi_{i}$ which are of classes $C^{\infty}$ then $\lambda_{j}$ $\in$ $C^{\infty}$ therefore the regularity condition is satisfied.\\
\textbf{2. The quadratic form $Q$ is dissipative}\\ To do this, it suffices to show that the quadratic form $Q(K^{w},K^{w}) \leq 0$.\\ Let us set
\begin{equation}\label{tws}
 Q(K^{w},K^{w})=T^{1}+T^{2}+T^{3}+T^{4}+T^{5}
\end{equation}
with \begin{center}
     $T^{1}  = \frac{1}{2}[g^{\delta\eta}]^{w}[g^{\lambda\mu}]^{w}(X_{\lambda\delta}X_{\alpha\eta}+X_{\lambda\delta}X_{\beta\eta}) $~,~
    $T^{2}   = -\frac{1}{4}[g^{\delta\eta}]^{w}[g^{\lambda\mu}]^{w}X_{\delta\lambda}X_{\mu\eta}$
   \end{center}
   \begin{center}
    $T^{3}   = -\frac{1}{2}[g^{\delta\eta}]^{w}[g^{\lambda\mu}]^{w} X_{\eta\lambda}(  X_{\mu\beta}+ X_{\mu\alpha}-X_{\alpha\beta}) $
   ~and~
     $T^{4}   = \frac{1}{4}[g^{\delta\eta}]^{w}[g^{\lambda\mu}]^{w}X_{\delta\eta}(Xg_{\alpha\lambda}+
X_{\beta\lambda}-X_{\alpha\beta})$\end{center}
where $X_{\lambda\delta}=[\partial_{0}g_{\lambda\delta}]^{w}$. Using the identity
  $ab=-\frac{1}{2}[(a-b)^{2}-a^{2}-b^{2}]$   which allows each product $ab$ to be transformed into a square so as to write $Q(K^{w},K^{w})$ as the sum of the negative quantities,
 we have: $ Q(K^{w},K^{w})
   = -\sum\limits C_{ijkl}(X_{ij}-X_{kl})^{2}
   \leq 0$
with $ C_{ijkl}=[g^{\delta\eta}]^{w}[g^{\lambda\mu}]^{w}\geq 0$, since the metric g has signature  $+---$. \\
\textbf{3. Energy estimate for the system}\\
Since system $(\ref{tyxwvt})$ is symmetrizable, we define the associated standard energy by:

 \begin{equation}\label{ml} E(t)=\frac{1}{2}\int_{\mathbb{R}^{3}}\sum\limits_{\alpha, \beta=0,...,3}\mid K^{w}_{\alpha \beta}(t,x)\mid ^{2}dx=\frac{1}{2}\int_{\mathbb{R}^{3}}\parallel K^{w}_{\alpha \beta}(t,x) \parallel^{2}dx \end{equation} So $E(t)$
 is equivalent to the norm $L^{2}$ on $\mathbb{R}^{3}$.\\
By differentiating the energy with respect to time we obtain:
\begin{equation}\label{xwq}
 \frac{dE(t)}{dt} = \int_{\mathbb{R}^{3}}
K^{w}_{\alpha \beta}(t,x)\partial_{t}K^{w}_{\alpha \beta}(t,x) dx
\end{equation}

 Since $[g^{00}]^{w}>0 $ and $\partial_{t}$
corresponds to i=0, we thus have:
\begin{eqnarray*}
 \partial_{t}K^{w}_{\alpha \beta}+\sum\limits_{i=1}^{3}(\frac{[g^{0i}]^{w}-(-1)^{w+1}\delta_{1j}[g^{ij}]^{w}}{[g^{00}]^{w}})\partial_{i}K^{w}_{\alpha \beta}  &=&- \frac{2}{[g^{00}]^{w}}F(K^{w}_{\alpha \beta})
\end{eqnarray*}  By replacing $\partial_{t}K^{w}_{\alpha \beta}$ with its value in $(\ref{xwq})$
we have:
\begin{eqnarray*}
 \frac{dE(t)}{dt} &=&-\int_{\mathbb{R}^{3}} K^{w}_{\alpha \beta} \sum\limits_{i=1}^{3}(\frac{[g^{0i}]^{w}-(-1)^{w+1}\delta_{1j}[g^{ij}]^{w}}{[g^{00}]^{w}})\partial_{i}K_{\alpha \beta}^{w}dx + \int_{\mathbb{R}^{3}} \frac{2}{[g^{00}]^{w}}F(K^{w}_{\alpha \beta})K_{\alpha \beta}^{w}dx
\end{eqnarray*}
By integration by parts, assuming that $K^{w} \rightarrow 0$ as $x \rightarrow \infty$, we obtain:
\begin{center}
$\int_{\mathbb{R}^{3}} K_{\alpha\beta} \sum\limits_{i=1}^{3}(\frac{[g^{0i}]^{w}-(-1)^{w+1}\delta_{1j}[g^{ij}]^{w}}{[g^{00}]^{w}})\partial_{i}K^{w}_{\alpha \beta}dx=0$
\end{center}
Hence
\begin{eqnarray*}
  \frac{dE(t)}{dt}
   &=& \int_{\mathbb{R}^{3}}\frac{2}{[g^{00}]^{w}} L(K^{w}_{\alpha \beta})K^{w}_{\alpha \beta}dx +  \int_{\mathbb{R}^{3}}\frac{2}{[g^{00}]^{w}} Q(K^{w}_{\alpha \beta},K^{w}_{\alpha \beta})K_{\alpha \beta}^{w}dx
\end{eqnarray*}
Since L is a linear and continuous functional in $K^{w}$
, and $Q$ is a dissipative form, there exists a constant $c>0$ such that:
 \begin{equation}\label{ycx}
\frac{dE(t)}{dt}\leq c \times E(t)
     \end{equation}By using Gr\"{o}nwall's lemma \cite{d}  we have:
     \begin{equation}\label{ycxt}
 \parallel E(t)\parallel \leq \parallel E(0)\parallel \times e^{c\times t}
     \end{equation}

Hence, the energy does not blow up instantaneously and remains bounded for finite time. Since $E(t)$ is equivalent to the      $L^{2}$-norm on  $\mathbb{R}^{3}$, the solution does not blow up for $t< +\infty$, which ensures stability and boundedness.
 Moreover, since $F$ is polynomial in $K^{w}$
, it is of class $C^{1}$
 and therefore locally Lipschitz continuous. By Theorem 1.1 in \cite{hb}, there exists a maximal solution $K^{w}:[ 0; T^{*})\rightarrow \mathcal{M}$. Arguing by contradiction, assume that  $T^{*}< +\infty$. Then, by Theorem 1 (Appendix A) in \cite{lc}, the solution $K^{w}$
leaves every compact subset of M as $t\rightarrow T^{*}$
, which is impossible since $K^{w}$
 remains bounded.
Therefore, the nonlinear first-order differential system $(\ref{tyr})$ admits global solutions.\par
Hence
$(\ref{tyr})$ is a hyperbolic hyper-quasilinear system of the
second order whose coefficients of the second derivatives are linear in the unknown therefore verify
the structure hypothesis $(\mathcal{G}_{0})$ and the system obtained after taking the restrictions on the characteristic hypersurfaces of the system $(\ref{tyr})$ admits
global solutions therefore the terms $Q_{\alpha\beta}(g,\partial g)$ verify the hypothesis of
structure $(\mathcal{G}_{1})$.\par
 We can therefore apply to the reduced system $(\ref{tyr})$
 the result of semi-global existence and uniqueness established in theorem $\ref{er}$.
 It will then remain to show that the solution of the reduced system is indeed solution of Einstein's equations of
empty. This gives rise to the following paragraph.
\begin{theorem}\label{rd}
Let $h^{w}_{22}$, $h^{w}_{23}$ and $h^{w}_{33}$ be functions
scalars of class $C^{\infty}$ on $S^{w}$ $(w=0,1)$ such that:

$(h_{\alpha \beta}^{w})=\left(
        \begin{array}{cc}
          h_{22}^{w} &  h_{23}^{w} \\
          h_{32}^{w} & h_{33}^{w} \\
        \end{array}
           \right)$
      either a symmetric matrix function defined
        positive with the determinant equal to 1 at each point of
        $S^{w}$ and $( h_{22}^{0}, h_{22}^{0}, h_{22}^{0})=( h_{22}^{1}, h_{22}^{1},
        h_{22}^{1})$ on $\Gamma$. Let $\tilde{\Omega}_{0}$, $\tilde{\Omega}_{1}$,
        $\tilde{b}_{12}$ and $\tilde{b}_{13}$ of class functions $C^{\infty}$ given on $\Gamma$.  \\
        There exists a unique scalar function $\Omega$ on $S^{0}\bigcup S^{1}$ and a
         Lorentz metric $g_{\alpha\beta}$ of
        class $C^{\infty}$ on $\mathcal{Y}$ such that:

        \begin{enumerate}
            \item $g_{\alpha\beta}=\Omega h_{\alpha\beta}$ on
            $S^{1}\bigcup S^{0}$ where $h_{\alpha\beta}=h^{0}_{\alpha\beta}$ on $S^{0}$ and
            $h_{\alpha\beta}=h^{1}_{\alpha\beta}$ sur $S^{1}$;
            \item  $(g_{\alpha\beta})$ statisfied Einstein's    equation for a vacuum in
         $\mathcal{Y}$;
           \item $(g_{\alpha\beta})$ and $\Omega$ induce the data on $S^{1}\bigcup
            S^{0}$;
            \item The coordinates on $\mathcal{M}$ are called standard for $g_{\alpha\beta}$ and the
            harmonic gauge condition $\Gamma^{\beta}=0$ satisfied on
            $\mathcal{Y}$;
            \item On $\Gamma$ we retain that $\Omega=\tilde{\Omega}$,
            $\Omega_{,0}=\tilde{\Omega}_{0}$,
            $\Omega_{,1}=\tilde{\Omega}_{1}$,
            $g_{02,1}=\tilde{b}_{12}$ and $g_{03,1}=\tilde{b}_{13}$.
        \end{enumerate}
\end{theorem}
\textbf{Proof:} It proceeds by following the approach used by A. Rendall \cite{Ren} for solving the initial constraint problem for the vacuum Einstein equations in harmonic gauge under $C^{\infty}$ regularity assumptions, with initial data prescribed on the union of two intersecting characteristic hypersurfaces. It also relies on the work of M. Dossa and C. Tadmon \cite{ty}, concerning the resolution of the initial constraint problem for the Einstein equations coupled with the Yang-Mills-Higgs system, formulated in harmonic gauge and Lorenz gauge, with initial data given on the union of two intersecting characteristic hypersurfaces.
\section{Discussion}By closely following the approach developed in this work, one could apply this result to obtain a semi-global resolution of the Goursat problem for the Einstein equations coupled with the Yang-Mills-Higgs system, formulated in harmonic gauge and Lorenz gauge, in weighted Sobolev-type spaces.
\section{Funding}Not applicable

\end{document}